\documentclass[11pt]{amsart}
\usepackage{amssymb, latexsym, amsmath, amsfonts, tikz, tikz-cd}
\usepackage{hyperref, enumitem, fancyhdr}

\newtheorem{thm}{Theorem}[section]
\newtheorem{cor}[thm]{Corollary}
\newtheorem{lem}[thm]{Lemma}
\newtheorem{prop}[thm]{Proposition}
\theoremstyle{definition}
\newtheorem{defn}[thm]{Definition}
\theoremstyle{remark}

\numberwithin{equation}{section}
\theoremstyle{remark}
\newtheorem{exam}[thm]{Example}
\theoremstyle{theorem}
\newtheorem{res}{Result}[section]

\newcommand{\mbb}{\mathbb}
\newcommand{\ra}{\rightarrow}

\newcommand{\sm}{\setminus}

\newcommand{\al}{\alpha}

\newcommand{\cal}{\mathcal}

\newcommand{\la}{\lambda}

\newcommand{\be}{\beta}

\begin{document}
\title{Riemann surface foliations with non-discrete singular set}
\keywords{Poincar\'{e} metric, singular holomorphic foliation, invariant analytic set}
%\thanks{The author is supported by the Labex CEMPI (ANR-11-LABX-0007-01)}
\subjclass{Primary: 32S65, 32M25  ; Secondary : 30F45}
\author{Sahil Gehlawat}

\address{Tata Institute of Fundamental Research,
Centre for Applicable Mathematics,
Bengaluru 560065, India.} 
\email{sahil24@tifrbng.res.in}

\begin{abstract} 
 Let $\mathcal F$ be a singular Riemann surface foliation on a complex manifold $M$, such that the singular set $E \subset M$ is non-discrete. We study the behavior of the foliation near the singular set $E$, particularly focusing on singular points that admit invariant submanifolds (locally) passing through them. Our primary focus is on the singular points that are removable singularities for some proper subfoliation. We classify singular points based on the dimension of their invariant submanifold and, consequently, establish that for hyperbolic foliations $\mathcal{F}$, the presence of such singularities ensures the continuity of the leafwise Poincar\'{e} metric on $M \sm E$.

\end{abstract}  

\maketitle 

\section{Introduction}

Let $M$ be a complex manifold of dimension $N$, and $E \subset M$ a closed analytic subset of codimension at least 2. A Riemann surface foliation $\cal{F}$ on $M$, singular along $E$, is defined by an atlas of flow-boxes $(U_{\al}, \phi_{\al})$ of $M \sm E$ where $\phi_{\alpha} : U_{\alpha} \ra \mbb D \times \mbb D^{N-1}$ ($\mbb D \subset \mbb C$ is the open unit disc). The transition maps $\phi_{\al\be}$ are holomorphic in $(x, y) \in \mbb D \times \mbb D^{N-1}$ and take the form
\[
\phi_{\al \be}(x, y) = \phi_{\al} \circ \phi^{-1}_{\be}(x, y) =  (P(x, y), Q(y)).
\]
The leaves of $\mathcal F$ are Riemann surfaces that are locally expressed as $\phi^{-1}_{\al}(\{y = \mbox{constant} \})$. The leaf passing through $p \in M\sm E$ is denoted by $L_{p}$. 

\medskip
A subset $\Sigma \subset M$ is said to be $\mathcal{F}-$invariant if for every $p \in \Sigma \sm E$, the corresponding leaf $L_{p} \subset \Sigma$. A subset $\Sigma \subset M$ is said to be local $\mathcal{F}-$invariant near $p \in \Sigma \subset M$ if there exists a neighborhood $U \subset M$ of $p$ such that $\Sigma$ is invariant under $\mathcal{F}\vert_{U}$, that is, $\Sigma$ is $\mathcal{F}\vert_{U}-$invariant. Our focus is on singular points $p \in E$ for which there exists a submanifold $S \subset M$ satisfying the following properties: $S$ is local $\mathcal{F}$-invariant near $p$, $S \not\subset E$ and $p$ is a removable singularity of the foliation $\mathcal{F}\vert_{S}$. The saturated foliation $\hat{\mathcal{F}}\vert_{S}$ corresponding to the foliation $\mathcal{F}\vert_{S}$ is obtained by including all the removable singularities of the subfoliation $\mathcal{F}\vert_{S}$, and $\hat{E}_{S} \subset E \cap S$ will be the singular set of $\hat{\mathcal{F}}\vert_{S}$.

\medskip
In \cite{CS}, Camacho and Sad proved a remarkable result for holomorphic foliations $\mathcal{F}$ on a complex surface $M$. They showed that for every singular point $p \in E$, there exists a separatrix passing through $p$. A separatrix through $p \in E$ is a local analytic curve $\mathcal{C} \subset M$, such that $p \in \mathcal{C}$, and $\mathcal{C}$ is local $\mathcal{F}-$invariant near $p$. In the above framework, this result implies that for $N=2$, there always exists an invariant submanifold $S \subset M$ of dimension 1 containing $p$. However, this property is no longer true in general for higher dimensions $(N > 2)$. As demonstrated in \cite{GL}, G\'{o}mez and Luengo gave an example of a Riemann surface foliation in $(\mathbb{C}^3,0)$ with discrete singular set $E = \{(0,0,0)\}$, where no separatrix exists through singular point $0$. In \cite{RR1}, J. Rebelo and H. Reis gave sufficient conditions for the existence of separatrices for a pair of germs of holomorphic vector fields $\{X,Y\}$ in $(\mathbb{C}^3,0)$. 
\medskip

The question of the existence of separatrices becomes less significant if the discreteness assumption on $E$ is dropped. When $p \in E$ is not an isolated singular point of $\mathcal{F}$, one can always construct an invariant analytic curve $\mathcal{C}$ passing through $p$. This follows from the facts that $E$ is an analytic subset and $\text{dim}(E) \ge 1$ at $p$, and one can take an analytic curve completely contained in $E$ passing through $p$.

\begin{defn}\label{D:Al}
    Let $\mathcal{F}$ be a singular Riemann surface foliation on a complex manifold $M$, with singular set $E$ of dimension $k$, where $1 \le k \le N-2$. For $1 \le l \le k$ and a local $\mathcal{F}-$invariant submanifold $\Sigma$ of dimension $l+1$, a singular point $p \in E \cap \Sigma$ is called a \textit{weakly removable singularity of order $l$ w.r.t. $\Sigma$}, if it satisfies the following conditions: 
    \begin{enumerate}
        \item $\Sigma \cap E$ is of dimension $l$,
        \item $p$ is a removable singularity of the subfoliation $\mathcal{F}\vert_{\Sigma}$, i.e., $p \notin \hat{E}_{\Sigma}$.
    \end{enumerate}
    The set of all \textit{weakly removable singularities of order $l$ w.r.t. $\Sigma$} is denoted by $A_{l, \Sigma}$. A point $p\in E$ is called a \textit{weakly removable singularity of order $l$} if $p \in A_{l,\Sigma}$ for some $\Sigma$. The collection of all such singularities is denoted by $A_{l}$. 
\end{defn}

\begin{defn}\label{D:A0}
     A singular point $p \in E$ is called a \textit{weakly removable singularity of order $0$} if there exists an analytic curve (local) $\mathcal{C}$ passing through $p$ such that $\mathcal{C} \sm \{p\}$ is contained in a leaf of $\mathcal{F}$. The set of all \textit{weakly removable singularities of order $0$} of the foliation $\mathcal{F}$ is denoted by $A_{0}$.
\end{defn}

\noindent \textbf{Remark:} From the above definition, it is evident that the points in $A_{0}$ exhibit a separatrix away from the singular set $E$. This property is stronger than simply requiring the existence of a separatrix, as we have already noted that separatrices always exist if $E$ has no $0$-dimensional component. This distinction also clarifies why the set $A_{0}$ is defined slightly differently from $A_{l}$ for $l \ge 1$. The motivation for this distinction lies in the study of the leafwise Poincar\'{e} metric. As shown in Proposition 1.10 of \cite{GV2}, the points in $A_{0}$ are precisely those that obstruct the continuity of this metric.

\begin{exam}
    Let $M = \mathbb{C}^3$ and $\mathcal{F}$ be the Riemann surface foliation on $\mathbb{C}^3$ induced by the holomorphic vector field
    \[
    X(x,y,z) = x\frac{\partial}{\partial x} + zy \frac{\partial}{\partial y} + 0 \frac{\partial}{\partial z}.
    \]
    The singular set of $\mathcal{F}$ is $E = \{(0,y,z) \in \mathbb{C}^3 : yz = 0\} = \{y-\text{axis}\} \cup \{z-\text{axis}\}$. Therefore, $E$ is $1$-dimensional and has no components of dimension $0$. Consider the hyperplane $\Sigma_{1} = \{x = 0\} \subset \mathbb{C}^3$. One can see that $\Sigma_{1}$ is invariant under $X$, and therefore $\Sigma_{1}$ is $\mathcal{F}-$invariant. Note that $\Sigma_{1} \cap E = E$ and 
    \[
    X\vert_{\Sigma_{1}}(0,y,z) = 0\frac{\partial}{\partial x} + zy \frac{\partial}{\partial y} + 0 \frac{\partial}{\partial z} = zy(0\frac{\partial}{\partial x} +  \frac{\partial}{\partial y} + 0 \frac{\partial}{\partial z}) = yz \tilde{X}(y,z),
    \]
    where $\tilde{X}(y,z) =  \frac{\partial}{\partial y} + 0 \frac{\partial}{\partial z}$ is a holomorphic vector field on $\Sigma_{1}$. By definition, the singular set of a foliation must be atleast codimension $2$, therefore the foliation $\mathcal{F}\vert_{\Sigma_{1}}$ on $\Sigma_{1}$ is induced by the vector field $\tilde{X}$. Since $\tilde{X}$ is non-vanishing, therefore every $p \in E = \Sigma_{1} \cap E$ is a removable singularity of $\mathcal{F}\vert_{\Sigma_{1}}$. Therefore, every $p \in E$ is a weakly removable singularity of order $1$ and thus $A_{1} = E$.

\medskip
    Now let $p = (0,c,0) \in E$. Consider the submanifold $\Sigma_{p} = \{(x,y,z) : z = 0, y = c\}$, and note that $\Sigma_{p}$ is $1-$dimensional, $\Sigma_{p} \cap E = \{p\}$ and is $\mathcal{F}-$invariant ($\Sigma_{p} \sm \{p\}$ is a leaf of $\mathcal{F}$). Therefore, $p$ is a weakly removable singularity of order $0$, that is, $p \in A_{0}$. Similarly we can consider the submanifold $\Sigma_{q} = \{(x,y,z) : y=0, z=c\}$ corresponding to the singular point $q = (0,0,c) \in E$. Thus, $A_{0} = A_{1} = E$.
\end{exam}

\begin{exam}
     Let $M = \mathbb{C}^3$ and $\mathcal{F}$ be the Riemann surface foliation on $\mathbb{C}^3$ induced by the holomorphic vector field
    \[
    X(x,y,z) = 2zy\frac{\partial}{\partial x} + 3x^2 \frac{\partial}{\partial y} + 0 \frac{\partial}{\partial z}.
    \]
    The singular set of $\mathcal{F}$ is $E = \{(0,y,z) \in \mathbb{C}^3 : yz = 0\} = \{y-\text{axis}\} \cup \{z-\text{axis}\}$. Therefore $E$ is $1$-dimensional and has no components of dimension $0$. Observe that $\Sigma_{0} = \{z = 0\} \subset \mathbb{C}^3$ is $\mathcal{F}-$invariant, and 
    \begin{equation}\label{E:Eq1}
    X\vert_{\Sigma_{0}}(x,y,0) = 0\frac{\partial}{\partial x} + 3x^2 \frac{\partial}{\partial y} + 0 \frac{\partial}{\partial z} = 3x^2 \tilde{X}(x,y),
    \end{equation}
    where $\tilde{X}(x,y) = 0\frac{\partial}{\partial x} +  \frac{\partial}{\partial y}$ is a holomorphic vector field defined on $\Sigma_{0}$. Since $\Sigma_{0} \cap E = \{x=z=0\} = \{y-\text{axis}\}$ and $\tilde{X}$ is non-vanishing, therefore $\{y-\text{axis}\} \subset A_{1}$. Now suppose if possible $(0,0,z_{0}) \in A_{1}$ for some $z_{0} \neq 0$. Therefore, by definition, there exists a local $\mathcal{F}-$invariant hypersurface $\Sigma$ such that $\Sigma \cap E \subset \{z-\text{axis}\}$ is 1-dimensional and $(0,0,z_{0}) \in \Sigma \cap E$ is a removable singularity of the subfoliation $\mathcal{F}\vert_{\Sigma}$. Also note that the hyperplane $\Sigma_{z_{0}} = \{z=z_{0}\}$ is $\mathcal{F}-$invariant and $\Sigma_{z_{0}} \cap E = \{(0,0,z_{0})\}$. On $\Sigma_{z_{0}}$, the foliation is induced by
    \[
    X\vert_{\Sigma_{z_{0}}}(x,y,z_{0}) = 2z_{0}y\frac{\partial}{\partial x} + 3x^2 \frac{\partial}{\partial y} + 0 \frac{\partial}{\partial z}.
    \]
    Observe that the foliation $\mathcal{F}\vert_{\Sigma_{z_{0}}}$ has exactly one separatrix passing through the singular point $(0,0,z_{0}) \in E$ which is given by the curve $\{(x,y,z_{0}) \in \Sigma_{z_{0}} : x^3 = z_{0} y^2\}$. Since $\Sigma \cap E \subset \{z-\text{axis}\}$ is 1-dimensional and $\Sigma_{z_{0}} \cap E = \{(0,0,z_{0})\}$, therefore $\Sigma \cap \Sigma_{z_{0}}$ will be 1-dimensional submanifold, and it would be $\mathcal{F}-$invariant. From the above discussion and the fact that $\Sigma \cap \Sigma_{z_{0}} \cap E = \{(0,0,z_{0})\}$, we get that the analytic curve $\mathcal{C}_{z_{0}} = \{(x,y,z_{0}) \in \mathbb{C}^3 : x^3 = z_{0} y^2, x\neq 0\}$ is a leaf of the foliation $\mathcal{F}\vert_{\Sigma}$. Since $\mathcal{C}_{z_{0}}$ is a separatrix corresponding to the singular point $(0,0,z_{0}) \in E$ and $\mathcal{C}_{z_{0}} \cup \{(0,0,z_{0})\}$ being singular at the point $(0,0,z_{0})$, we get that $(0,0,z_{0})$ cannot be a removable singularity of the subfoliation $\mathcal{F}\vert_{\Sigma}$. This is a contradiction to our above assumption. Hence $A_{1} = \{y-\text{axis}\} \subset E$.

    Note that $\{z-\text{axis}\}^{\ast} = \{(0,0,z) \vert \ z \neq 0\} \subset A_{0}$, since for each $p = (0,0,z_{0}) \in \{z-\text{axis}\}^{\ast}$, the set $\mathcal{C} = \{(x,y,z_{0}) \in \mathbb{C}^3 : x^3 = z_{0} y^2, x\neq 0\}$ is an analytic curve such that $\mathcal{C} \cap E = \{p\}$ and $\mathcal{C}\sm \{p\}$ is a leaf of $\mathcal{F}$. We claim that $A_{0} = \{z-\text{axis}\}^{\ast}$. Suppose if possible $q = (0,y,0) \in A_{0}$. By definition, there exists an invariant analytic curve $\mathcal{C}_{q}$ such that $\mathcal{C}_{q} \cap E = \{q\}$. But since $\Sigma_{z_{0}} = \{z = z_{0}\}$ is invariant for each $z_{0}$ and $q \in \Sigma_{0}$, therefore $\mathcal{C}_{q} \subset \Sigma_{0}$. We already saw from $\ref{E:Eq1}$ that the leaves of $\mathcal{F}\vert_{\Sigma_{0}}$ are given by $L = \{(x,y,0) \vert \ x =  \text{cste}\}$. Therefore $\mathcal{C}_{q} = \{(x,y,0) \vert \ x=0\} \subset E$. But this is a contradiction since $\mathcal{C}_{q} \cap E = \{q\}$. Thus $A_{0} = \{z-\text{axis}\}^{\ast} = E \sm A_{1}$.
    
\end{exam}

%\medskip

\begin{defn}\label{D:B_l}
 Let $\mathcal{F}$ be a singular Riemann surface foliation on a complex manifold $M$, with singular set $E$ of dimension $k$, where $1 \le k \le N-2$. For $0 \le l \le k$ and a local $\mathcal{F}-$invariant submanifold $\Sigma$ of dimension $l+1$, a singular point $p \in E \cap \Sigma$ is called a \textit{strongly removable singularity of order $l$ w.r.t. $\Sigma$}, if it satisfies the following conditions: 
    \begin{enumerate}
        \item $\Sigma \cap E$ is of dimension $l$,
        \item $p$ is a removable singularity of the subfoliation $\mathcal{F}\vert_{\Sigma}$, i.e., $p \notin \hat{E}_{\Sigma}$,
        \item the leaf $\hat{L}_{p}$ of the saturated foliation $\hat{\mathcal{F}}\vert_{\Sigma}$ satisfies $\hat{L}_{p} \subset E$.
    \end{enumerate}
    The set of all \textit{strongly removable singularity of order $l$ w.r.t. $\Sigma$} is denoted by $B_{l, \Sigma}$. A point $p\in E$ is called a \textit{strongly removable singularity of order $l$} if $p \in B_{l,\Sigma}$ for some $\Sigma$. The collection of all such singularities is denoted by $B_{l}$.
    
\end{defn}

%\noindent One can note from the above definition that $B_{0} = \emptyset$, and $B_{l} \subset A_{l}$ for each $0 \le l \le k$.

\noindent Let $(M, \mathcal{F},E)$ be a singular Riemann surface foliation on a complex manifold $M$ and the singular set $E$ is of dimension $k$, where $1 \le k \le N-2$. For $0 \le m \le k$, define $E_{m} \subset E$ as the union of all $m-$dimensional irreducible components of $E$. Note that $E = \cup_{m=0}^{k} E_{m}$.

\begin{thm}\label{T:Removable Singular sets}
     Let $\mathcal{F}$ be a singular Riemann surface foliation on a complex manifold $M$, with singular set $E$ of dimension $k$, where $0 \le k \le N-2$. Then
     \begin{enumerate}
         \item $B_{0} = \emptyset$ and $B_{l} \subset A_{l}$, for all $0 \le l \le k$,
         \item $A_{l}, B_{l} \subset \cup_{m =l}^{k} E_{m}$, for all $1 \le l \le k$,
         \item $A_{l} \subset A_{l-1}$, for all $2 \le l \le k$,
         \item $B_{l} \subset B_{l-1}$, for all $2 \le l \le k$,
         \item $A_{l} \sm B_{l} \subset A_{0}$, for all $1 \le l \le k$.
     \end{enumerate}
\end{thm}

\noindent The following examples illustrate that the inclusions in the above result could very well be proper.

\begin{exam}\label{E:Ex1.7}
 Let $M = \mathbb{C}^4$ and $\mathcal{F}$ be the singular Riemann surface foliation on $\mathbb{C}^4$ induced by the holomorphic vector field
    \[
    X(x,y,z,w) = x\frac{\partial}{\partial x} + yz \frac{\partial}{\partial y} + yw \frac{\partial}{\partial z} +  xw \frac{\partial}{\partial w}.
    \]
    The singular set of $\mathcal{F}$ is $E =E_1 \cup E_2$, where $E_{1} = \{(x,y,z,w) \vert x=z=w=0\} = \{y-\text{axis}\}$, and $E_{2} = \{(x,y,z,w) \vert x=y=0\} = \{zw-\text{plane}\}$.
    \medskip

    \begin{enumerate}

    \item Consider $\Sigma_{xw} = \{x=w=0\}$. Note that $\Sigma_{xw}$ is $\mathcal{F}-$invariant, $E \cap \Sigma_{xw} = \{z-\text{axis}\} \cup \{y-\text{axis}\}$ and 
    \[
    X\vert_{\Sigma_{xw}}(0,y,z,0) = 0\frac{\partial}{\partial x} + yz \frac{\partial}{\partial y} + 0 \frac{\partial}{\partial z} +  0 \frac{\partial}{\partial w} = yz \frac{\partial}{\partial y}.
    \]
    Therefore, the saturated foliation $\hat{\cal{F}}\vert_{\Sigma_{xw}}$ on $\Sigma_{xw}$ is induced by the vector field $\hat{X}_{\Sigma_{xw}} \equiv \frac{\partial}{\partial y}$, which in turn gives us $\{z-\text{axis}\} \cup \{y-\text{axis}\} \subset A_{1}$, $\{z-\text{axis}\}^{\ast} \subset A_{0}$, and $\{y-\text{axis}\} \subset B_{1}$.
    %\medskip

    \item Next we look at $\Sigma_{xc} = \{x=0,w=c\}$ for $c \neq 0$, and check that it is $\mathcal{F}-$invariant, $E \cap \Sigma_{xc} = \{(0,0,z,c) \vert z \in \mathbb{C}\}$, and 
    \[
    X\vert_{\Sigma_{xc}}(0,y,z,c) = 0\frac{\partial}{\partial x} + yz \frac{\partial}{\partial y} + yc \frac{\partial}{\partial z} +  0 \frac{\partial}{\partial w} = y \left(z \frac{\partial}{\partial y} + c \frac{\partial}{\partial z}\right).
    \]
    Again by looking at the saturated foliation in this case we get $(0,0,z,c) \in A_{0} \cap A_{1}$ for all $z \in \mathbb{C}$. Therefore $A_{1} = E$, and $E_{2} \sm \{0\} \subset A_{0}$. 
    
    \item Consider $\Sigma_{zw} = \{z=w=0\}$ and note that $\Sigma_{zw}$ is $\mathcal{F}-$invariant, $E \cap \Sigma_{zw} = \{(0,y,0,0) \vert y \in \mathbb{C}\}$, and 
    \[
    X\vert_{\Sigma_{zw}}(x,y,0,0) = x\frac{\partial}{\partial x} + 0 \frac{\partial}{\partial y} + 0 \frac{\partial}{\partial z} +  0 \frac{\partial}{\partial w} = x\frac{\partial}{\partial x}.
    \]
    By looking at the saturated foliation $\hat{\mathcal{F}}\vert_{\Sigma_{zw}}$, we get $\{y-\text{axis}\} = E_1 \subset A_{0}$ and therefore $A_{0} = E$. 
    %%\medskip

    \item By Theorem \ref{T:Removable Singular sets}, $A_2 \subset E_{2} \subsetneq E = A_{1}$, which gives $A_{2} \subsetneq A_{1}$. If we look at $\Sigma_{y} = \{y=0\}$, and check that it is $\mathcal{F}-$invariant, $E \cap \Sigma_{y} = E_{2}$, and
     \[
    X\vert_{\Sigma_{y}}(x,0,z,w) = x\frac{\partial}{\partial x} + 0 \frac{\partial}{\partial y} + 0 \frac{\partial}{\partial z} +  xw \frac{\partial}{\partial w} = x \left(\frac{\partial}{\partial x} + w\frac{\partial}{\partial w}\right).
    \]
    The saturated foliation $\hat{\mathcal{F}}\vert_{\Sigma_{y}}$ satisfies $\text{sing}(\hat{\mathcal{F}}\vert_{\Sigma_{y}}) = \emptyset$, which gives $E_{2} \subset A_{2}$. Thus $A_{2} = E_{2}$.

    \item Next we show that $\emptyset \neq B_{1} \subsetneq A_{1}$. We already saw that $E_{1} \subset B_{1}$, therefore $B_{1} \neq \emptyset$. For the other observation, consider a singular point $p = (0,0,z,w) \in E$ such that $z \neq 0$. We will prove that $p \not\in B_{1}$. To see this, we claim that it suffices to show that $\mathcal{F}$ is transversal type at $p$ (see the paragraph before Theorem \ref{T:Cont-Extn} for the definition). Observe that if $p \in B_{1}$, we can use the local structure of the subfoliation to conclude that 
    \[
    C_{p}{E} \cap \overline{C_{p}{\mathcal{F}}} \neq \{0\}.
    \]
    Therefore, $\mathcal{F}$ cannot be transversal type at $p \in B_{1} \subset E$. So our claim is verified.
    
    Now note that $C_{p}E = \{(0,0,z,w) \in \mathbb{C}^4\} = \text{Span}_{\mathbb{C}}\{(0,0,1,0), (0,0,0,1)\}$. Let $V = (0,0,v_{3}, v_{4}) \in C_{p}E \cap C_{p}\mathcal{F}$ and $p_{n} = (x_n, y_n, z_n, w_n) \in \mathbb{C}^4 \sm E$ be a sequence of non-singular points such that $p_n \to p$ and $V_{n} \in T_{p_{n}}\mathcal{F}$ be such that $V_{n} \to V$. Since the foliation $\mathcal{F}$ is defined by the holomorphic vector field $X$, there exist $\alpha_{n} \in \mathbb{C}$ such that $V_{n} = \alpha_{n}  X(p_{n}) = \alpha_{n}  (x_n, y_n z_n, y_n w_n, x_n w_n)$. Therefore we get 
    \[
    \alpha_{n}  x_n \to 0, \ \ \alpha_{n}  y_{n} z_{n} \to 0, \ \ \alpha_{n}  y_{n} w_{n} \to v_{3}, \ \ \text{and} \ \ \alpha_{n} x_{n} w_{n} \to v_{4}.
    \]
    Since $\alpha_{n} x_{n} \to 0$, and $w_{n} \to w$, therefore $\alpha_{n} x_{n} w_{n} \to 0$, that is $v_{4} = 0$. Also since $z_{n} \to z \neq 0$ and $\alpha_{n} y_{n} z_{n} \to 0$ gives us $\alpha_{n} y_{n} \to 0$. Using $w_{n} \to w$, we get $\alpha_{n} y_{n} w_{n} \to 0$ and $v_{3} = 0$. Therefore $V = \vec{0}$ and $C_{p}E \cap C_{p}\mathcal{F} = \{0\}$. Thus, $\mathcal{F}$ is transversal type at $p$, and therefore 
    \[
    B_{1} \subset \{y-\text{axis}\} \cup \{w-\text{axis}\} \subsetneq A_{1}.
    \]

    \item To see $B_{2} \subsetneq B_{1}$, observe that by Theorem \ref{T:Removable Singular sets}, $B_{2} \subset E_{2} \cap B_{1} \subset \{w-\text{axis}\}$ and by step $(1)$ in this example $\{y-\text{axis}\} \subset B_{1}$.

    \end{enumerate}    
\end{exam}

\medskip

We now demonstrate the significance of the strongly removable singular points ($B_{l}$) in the framework of singular hyperbolic foliations. In particular, we will use them to study the leafwise Poincar\'{e} metric of the corresponding hyperbolic foliation. Recall that a singular holomorphic foliation $\mathcal{F}$ on $M$ is called \textit{hyperbolic} if every leaf of $\mathcal{F}$ is a hyperbolic Riemann surface. Let $g$ be a Hermitian metric on the complex manifold $M$, and denote the length of a tangent vector $v$ with respect to $g$ by $\vert v\vert_{g}$. Define $\mathcal{O}(\mbb{D}, M\sm E)$ as the set of holomorphic maps from the unit disc $\mbb D$ into $M \sm E$. Consider the subset $\mathcal{O}(\mbb{D}, \cal{F}) \subset \mathcal{O}(\mbb{D}, M\sm E)$, consisting of holomorphic maps $f$ such that the image $f(\mbb{D})$ lies entirely within the leaf $L_{f(0)}$ of $\mathcal{F}$. The modulus of uniformization map $\eta : M \sm E \ra (0, \infty)$ is defined as
\[
\eta(p) := \sup \left \{ \vert f'(0) \vert_g : f \in  \mathcal O(\mbb D, \mathcal F), f(0) = p \right\}.
\]

\noindent It is straightforward to verify that $\eta > 0$, and the supremum is achieved by the universal covering map $\pi_p : \mbb D  \ra L_p$, which satisfies $\pi_p(0) = p$. Consequently, $\eta(p) = \vert \pi'_p(0) \vert_g$. For a leaf $L \subset \cal{F}$ and a tangent vector $v$ at a point $p \in L$, let $\la_{ L}$ denote the Poincar\'{e} metric on $L$ and $\vert v \vert_{L}$ denote the Poincar\'{e} length of $v$. By the extremal property of the Kobayashi metric, we have 
\begin{equation}\label{E:Metric Relation}
\vert v \vert_g = \eta(p) \vert v \vert_{L}.
\end{equation}

\noindent Hence, the restriction of the metric $g/{\eta}$ to a leaf $L$ naturally induces the Poincar\'{e} metric $\lambda_{L}$ on that leaf.

%\medskip

Note that, since the Poincar\'{e} metric $\la_{L}$ for a leaf $L$ is positive real-analytic on $L$ and $g$ is smooth on $M$, therefore $\eta$ is also smooth along leaves. However, the regularity of $\eta$ along transverse directions is not immediately apparent. In fact, from the above observation, the regularity of $\eta$ is connected to the variation of $\la_{L}$ along such directions. The regularity question of $\eta$ under various hypotheses on $M, E$ and $\mathcal F$ has been studied by Verjovsky \cite{V}, Lins Neto \cite{N1, N2}, Candel \cite{Ca} and Fornaess-Sibony \cite{FS}. A stronger estimate on the modulus of continuity of $\eta$ for hyperbolic foliations on $\mathbb{CP}^n$ was recently obtained by Dinh--Nguy\^{e}n--Sibony \cite{DNS1, DNS2} and Bacher \cite{FB} under suitable hypotheses on $E$. The significance of studying the transverse regularity of the leafwise Poincar\'{e} metric has been emphasized in the works of Nguyen (see \cite{Ng1, Ng2}), where it has been applied to investigate the ergodic properties of Riemann surface laminations.

\medskip

Much of the prior work in this area assumes the singular set $E$ to be discrete. In recent studies \cite{G1, GV2}, we explored the problem without imposing the discreteness assumption on the singular set. In \cite{GV2}, we provided a sufficient condition for the continuity of the modulus of the uniformization map $\eta$ on the non-singular set $M \sm E$, building on and strengthening a result by Fornaess–Sibony for the case where the singular set is discrete (see Theorem 20, \cite{FS}). Additionally, we discussed the possibility of a continuous extension of $\eta$ to the singular set $E$.

Our investigations revealed significant differences when the discreteness assumption on $E$ is dropped. For instance, in \cite{GV2}, we observed that the continuity of $\eta$ on the non-singular part $M \sm E$ does not necessarily imply the existence of a continuous extension to $E$, unlike the case when $E$ is discrete. Furthermore, we identified sufficient conditions on the singular set, referred to as being of \textit{transversal type}, which ensure the continuous extension of $\eta$ on $E$. These findings underscore the nuanced relationship between the structure of the singular set and the behavior of $\eta$. 

 %\medskip
 The aim of this note is two fold. First, we utilize the strongly removable singular points $(B_{l})$ to obtain sufficient conditions for the continuity of $\eta$ on $M \sm E$, in the realm of non-discrete singular set $E$. Second, we improve the observations obtained in \cite{GV2} about the continuous extension of $\eta$ on $E$. Specifically, we prove that if the \textit{transversal type} property holds in a sufficient large subset of $E$, then we have a continuous extension of $\eta$ to whole $M$. 

\noindent Before presenting the first result, let us recall an essential object relevant to this study. For a leaf $L$ of a hyperbolic foliation $\mathcal{F}$, a holomorphic map $\phi: \mathbb{D} \to L$ is called a uniformization of $L$ if it is a covering map. As discussed in \cite{NM}, the collection of all uniformizations
\[
\mathcal{U} = \{\alpha \in \mathcal{O}(\mbb{D}, \mathcal{F}) \mid \alpha \ \text{is a uniformization of a leaf of} \ \mathcal{F}\},
\]
plays a crucial role in the analysis of foliations. The space $\mathcal U$ is said to satisfy NCP (Normal on compact parts) property if, for any subfamily $\mathcal{H} \subset \mathcal{U}$ with the set $\{\alpha(0) \mid \alpha \in \mathcal{H}\}$ being relatively compact in $M$ is itself a normal family. This property is guaranteed, for instance, when $M$ is a taut manifold.

Let $(M,\mathcal{F}, E)$ be a singular hyperbolic foliation on the complex manifold $M$. For $p \in B_{1}$, define
\[
l_{p} := \text{max}\{j \in \mathbb{N} \vert \ p \in B_{j}\}.
\]
Note that since $B_{i} \subset B_{j} \subset B_{1}$ for all $i \ge j \ge 1$, therefore $B_{j} = \{p \in B_{1} \vert \ l_{p} \ge j\}$.

\begin{thm}\label{T: Regularity of metric}
    Let $\mathcal{F}$ be a singular hyperbolic foliation on a complex manifold $M$ and singular set $E$. If \ $\mathcal{U}$ is NCP, $A_{0} \subset B_{1}$, and for each $p \in A_{0}$, there exists $(N-1 - l_{p})$ number of local $\mathcal{F}$-invariant hypersurfaces $\Sigma_{1,p}, \Sigma_{2,p}, \ldots, \Sigma_{N-1 -l_{p},p}$ such that 
    \[
    \Sigma_{p} = \cap_{i =1}^{N-1-l_{p}} \Sigma_{i,p}
    \]
    is a $(l_{p} +1)-$dimensional submanifold and $p \in B_{l_{p}, \Sigma_{p}}$. Then $\eta$ is continuous on $M \sm E$.
\end{thm}

\noindent The following corollary is immediate, with $l_{p} = N-2$ for all $p \in E$.

\begin{cor}\label{C: Regularity of metric}
     Let $\mathcal{F}$ be a singular hyperbolic foliation on a complex manifold $M$ and singular set $E$. If \ $\mathcal{U}$ is NCP, and $A_{0} \subset B_{N-2}$, then $\eta$ is continuous on $M \sm E$.
\end{cor}

\begin{exam}
    Let $M = \mathbb{D}^4 \subset \mathbb{C}^4$ and $\mathcal{F}$ be the singular hyperbolic foliation on $M$ induced by the holomorphic vector field
    \[
    X(x,y,z,w) = y e^{y}\frac{\partial}{\partial x} + y e^{w} \frac{\partial}{\partial y} + z \frac{\partial}{\partial z} + x \frac{\partial}{\partial w}.
    \]
    The singular set of $\mathcal{F}$ is $E = \{(0,0,0,w) \vert \ w \in \mathbb{D}\}  = \{w-\text{axis}\} \cap M = E_{1}$. Since $E_{2} = \emptyset$, therefore $B_{2} = \emptyset$ as well. Note that the hypersurfaces $\Sigma_{y} = \{y = 0\}$ and $\Sigma_{z} = \{z = 0\}$ are $\mathcal{F}-$invariant. Consider $\Sigma_{yz} = \Sigma_{y} \cap \Sigma_{z} = \{y=z=0\}$. It is easy to check that $\Sigma_{yz}$ is an $\mathcal{F}-$invariant submanifold of $M$ of dimension $2$, and the subfoliation $\mathcal{F}\vert_{\Sigma_{yz}}$ is induced by
    \[
    X\vert_{\Sigma_{yz}}(x,0,0,w) = x \frac{\partial}{\partial w}.
    \]
    Clearly the saturated subfoliation $\hat{\mathcal{F}}\vert_{\Sigma_{yz}}$ is non-singular, and in fact $\{(0,0,0,w) \vert \ w \in \mathbb{D}\}$ is a leaf of $\hat{\mathcal{F}}\vert_{\Sigma_{yz}}$. Therefore $E \subset B_{1, \Sigma_{yz}}$, which in turn gives that $A_{0} \subset B_{1} = E$. Now one can use Theorem \ref{T: Regularity of metric} to deduce the continuity of $\eta$ on $M \setminus E$.
\end{exam}

\medskip
\noindent The following example is motivated from \cite{RR2} (see Theorem 1).
\begin{exam}
     Let $M = \mathbb{D}^3 \subset \mathbb{C}^3$ and $\mathcal{F}$ be the singular hyperbolic foliation on $M$ induced by the holomorphic vector field
    \[
    X(x,y,z) = (P(y) + z \tilde{f}(x,y,z)) \frac{\partial}{\partial x} + z \tilde{g}(x,y,z) \frac{\partial}{\partial y} + z^n \frac{\partial}{\partial z},
    \]
    where $P(y)$ is a polynomial of degree $k \ge 1$, satisfying $P(0) = 0$. Suppose $S = \{0, c_{1}, c_{2}, \ldots, c_{k-1}\}$ be the set of zeroes of $P$. The singular set of $\mathcal{F}$ is 
    \[
    E = \{(x,0,0) \vert \ x \in \mathbb{D}\} \cup (\bigcup_{i, \vert c_{i}\vert < 1} \{(x,c_{i},0) \vert \ x \in \mathbb{D}\}).
    \]
    Consider the $\mathcal{F}-$invariant hypersurface $\Sigma_{z} = \{z=0\}$. On $\Sigma_{z}$,
    \[
    X\vert_{\Sigma_{z}}(x,y,0) = P(y) \frac{\partial}{\partial x}.
    \]
    It is now clear that the saturated subfoliation $\hat{\mathcal{F}}\vert_{\Sigma_{z}}$ is non-singular on $\Sigma_{z}$, and $\{(x,c,0) \vert \ x \in \mathbb{D}\}$ is a leaf of $\hat{\mathcal{F}}\vert_{\Sigma_{z}}$ for all $c \in S \cap \mathbb{D}$. Therefore $B_{1} = E$, and by Corollary \ref{C: Regularity of metric} we get the continuity of $\eta$ on $M \sm E$.
\end{exam}

\medskip
\begin{exam}
    Let $M = \mathbb{D}^N \subset \mathbb{C}^N$ and $\mathcal{F}$ be the singular hyperbolic foliation on $M$ induced by the holomorphic vector field
    \[
    X(x_{1},x_{2}, \ldots, x_{N}) = x^{R_{1}} x_{N}^{-1} \frac{\partial}{\partial x_{1}} + \sum_{i=2}^{N}{x^{R_{i}} x_{i-1}^{-1} \frac{\partial}{\partial x_{i}}},
    \]
    where $x=(x_{1}, x_{2}, \ldots, x_{N}) \in M$, $R_{i} = (a_{1,i}, a_{2,i}, \ldots, a_{N,i}) \in \mathbb{N}^N$ be such that $a_{i-1,i} = 1$ for $i \ge 2$ and $a_{N,1} =1$, and $x^{R_{i}} = \prod_{j=1}^{N}{{x_{j}}^{a_{j,i}}}$. The singular set of $\mathcal{F}$ is $E = \cup_{i \neq j}\{x_{i} = x_{j} = 0\} = E_{N-2}$. Check that $\Sigma_{i} = \{x_{i} = 0\}$ is $\mathcal{F}-$invariant for all $1 \le i \le N$, and the subfoliation $\mathcal{F}\vert_{\Sigma_{i}}$ for $1 \le i \le N-1$, is induced by
    \[
     X\vert_{\Sigma_{i}}(x_{1},x_{2}, \ldots, x_{i-1}, 0, x_{i+1}, \ldots x_{N}) = x^{R_{i+1}} x_{i}^{-1} \frac{\partial}{\partial x_{i+1}}.
    \]
    It is evident from above that the saturated subfoliation $\hat{\mathcal{F}}\vert_{\Sigma_{i}}$ is non-singular, and in fact $\{x_{i} = x_{j} = 0\} \subset B_{N-2, \Sigma_{i}} \subset B_{N-2}$ for all $j \neq i, i+1$. But with the same argument $\{x_{i} = x_{i+1} = 0\} \subset B_{N-2, \Sigma_{i+1}}$.

    \noindent Similar calculations on $\Sigma_{N}$ will give us $E = B_{N-2}$, and then Corollary \ref{C: Regularity of metric} give us the continuity of $\eta$ on $M \sm E$. 
\end{exam}
\medskip

The next theorem establishes the continuity of the modulus of uniformization map $\eta$ on $M \sm E$ by leveraging Theorem \ref{T: Regularity of metric} and Corollary 1.13 in \cite{GV2}. To achieve this, we introduce two key properties of singular points that have been used in the above mentioned results. We say
\begin{itemize}
    \item $p \in E$ satisfies property $(L)$ if $p \in B_{1}$, and there exist $(N-1 - l_{p})$ number of local $\mathcal{F}$-invariant hypersurfaces $\Sigma_{1,p}, \Sigma_{2,p}, \ldots, \Sigma_{N-1 -l_{p},p}$ such that 
    \[
    \Sigma_{p} = \cap_{i =1}^{N-1-l_{p}} \Sigma_{i,p}
    \]
    is a $(l_{p} +1)-$dimensional submanifold and $p \in B_{l_{p}, \Sigma_{p}}$,

    \item $p \in E$ satisfies property $(M)$ if there exist $(N-k)$ number of local $\mathcal{F}-$invariant hypersurfaces $\tilde{\Sigma}_{1,p}, \tilde{\Sigma}_{2,p}, \ldots, \tilde{\Sigma}_{N-k,p}$ such that 
    \[
    p \in \cap_{i=1}^{N-k}{\tilde{\Sigma}_{i,p}} \subset E,
    \]
    where $k = \text{dim}_{p}E$.
\end{itemize}
We define $E_{L} := \{p \in E : p \ \text{satisfies property} \ (L)\}$ and $E_{M} := \{p \in E : p \ \text{satisfies property} \ (M)\}$. Note that the continuity of $\eta$ is ensured if either $A_{0} \subset E_{L}$ (by Theorem \ref{T: Regularity of metric}) or $E = E_{M}$ (by Corollary 1.13 in \cite{GV2}).

\begin{thm}\label{T:Mix Regularity}
     Let $\mathcal{F}$ be a singular hyperbolic foliation on a complex manifold $M$ and singular set $E$. If \ $\mathcal{U}$ is NCP, and $A_{0} \subset E_{L} \cup E_{M}$, then the modulus of uniformization map $\eta$ is continuous on $M \sm E$.
\end{thm}

\begin{exam}
    Let $M = \mathbb{D}^4 \subset \mathbb{C}^4$ and $\mathcal{F}$ be the singular hyperbolic foliation on $M$ induced by the holomorphic vector field
    \[
    X(x,y,z,w) = x\frac{\partial}{\partial x} + yz \frac{\partial}{\partial y} + yw \frac{\partial}{\partial z} +  xw \frac{\partial}{\partial w}.
    \]
    The singular set of $\mathcal{F}$ is $E =E_1 \cup E_2$, where $E_{1} = \{(x,y,z,w) \in M \vert x=z=w=0\} = \{y-\text{axis}\}$, and $E_{2} = \{(x,y,z,w) \in M \vert x=y=0\} = \{zw-\text{plane}\}$. 
    
    \noindent We already observed in Example \ref{E:Ex1.7}, that $A_{0} = E$, and $E_{1} \subset B_{1} \subset E_{1} \cup \{(0,0,0,w) \vert w \in \mathbb{D}\}$. For $p \in E_{1}$, we can check that $l_{p} = 1$ and $p \in B_{1, \Sigma_{xw}}$, where $\Sigma_{xw} = \{x=w=0\} = \{x=0\} \cap \{w=0\} = \Sigma_{x} \cap \Sigma_{w}$. Since $\Sigma_{x}, \Sigma_{w}$ are $\mathcal{F}-$invariant, therefore $E_{1} \subset E_{L}$.

    \noindent Recall that $\text{dim}_{p}E = 2$ for $p \in E_{2}$, and $E_{2} = \{x=0\} \cap \{y=0\} = \Sigma_{x} \cap \Sigma_{y}$. Once again, using the fact that $\Sigma_{x}, \Sigma_{y}$ are $\mathcal{F}-$invariant, we get $E_{2} \subset E_{M}$.

    \noindent Therefore we get $A_{0} = E = E_{1} \cup E_{2} \subset E_{L} \cup E_{M}$, that is 
    \[
    A_{0} = E_{L} \cup E_{M}.
    \]
    Thus, by Theorem \ref{T:Mix Regularity}, the map $\eta$ is continuous on $M \sm E$.
\end{exam}

\medskip
Next we talk about the extension of the map $\eta$ to the singular set $E$. In \cite{GV2} (Theorem 1.8), it was proved that being \textit{transversal type} is a sufficient condition for the continuous extension of $\eta$ to the singular set. Recall that, a foliation $\mathcal{F}$ is transversal type at a singular point $p \in E$ if there exists a neighborhood $U_{p}$ of $p$ such that for all $q \in U_{p} \cap E$, 
\[
\overline{C_{q}\mathcal{F}} \cap C_{q}E = \{0\},
\]
where $C_{p}E = \{v\in T_{p}M :  \text{there exists }  \{q_{n}\}_{n \ge1} \subset \text{reg}(E), \ v_{n} \in T_{q_n}E \ \text{such that} \ (q_{n}, v_n) \to (p,v)\}$ is the Whitney's $C_{4}-$tangent cone (see \cite{Ch}), and $C_{p}\mathcal{F} = \{v\in T_{p}M :  \text{there exists }  \{q_{n}\}_{n \ge1} \subset M \sm E, \ v_{n} \in T_{q_n}\mathcal{F} \ \text{such that} \ (q_{n}, v_n) \to (p,v)\}$.

\begin{thm}\label{T:Cont-Extn}
    Let $\mathcal{F}$ be a singular hyperbolic foliation on a complex manifold $M$ and singular set $E$. Suppose that $\mathcal{U}$ is NCP and $\eta$ is continuous on $M \sm E$. If the set $E \sm \{p \in E \vert \ \mathcal{F} \ \text{is transversal type at} \ p\}$ is discrete, then $\eta$ has a continuous extension to whole $M$.
\end{thm}

%\medskip
\begin{exam}
     Let $M = \mathbb{D}^3 \subset \mathbb{C}^3$ and $\mathcal{F}$ be the singular hyperbolic foliation on $M$ induced by the holomorphic vector field
    \[
    X(x,y,z) = x(y+z)\frac{\partial}{\partial x} + y(x+z) \frac{\partial}{\partial y} + z(x+y) \frac{\partial}{\partial z}.
    \]
    The singular set of $\mathcal{F}$ is $E = \{(x,0,0) \vert \ x \in \mathbb{D}\} \cup \{(0,y,0) \vert \ y \in \mathbb{D}\} \cup \{(0,0,z) \vert \ z \in \mathbb{D}\}$. Using the fact that $\Sigma_{x} = \{x=0\}, \Sigma_{y} = \{y=0\}$ and $\Sigma_{z} = \{z=0\}$ are $\mathcal{F}-$invariant and $\{(x,0,0) \vert \ x \in \mathbb{D}\} = \Sigma_{y} \cap \Sigma_{z}$, $\{(0,y,0) \vert \ y \in \mathbb{D}\} = \Sigma_{x} \cap \Sigma_{z}$ and $\{(0,0,z) \vert \ z \in \mathbb{D}\} = \Sigma_{x} \cap \Sigma_{z}$, we get $E = E_{M}$. Therefore, Theorem \ref{T:Mix Regularity} gives us the continuity of $\eta$ on $M \sm E$. 
    
    Next we prove that $\mathcal{F}$ is transversal type at each $p \in E \sm \{0\}$. Let $p = (x,0,0) \in E$ be such that $x \neq 0$ and $V = (v_{1},v_{2}, v_{3}) \in C_{p}E$. Since $C_{p}E = \text{Span}_{\mathbb{C}}\{(1,0,0)\}$, therefore $v_{2} = v_{3} = 0$. Suppose $p_{n} = (x_{n}, y_{n}, z_{n}) \in M \sm E$ and $V_{n} \in T_{p_n}\mathcal{F}$ be such that $(p_{n}, V_{n}) \to (p,V)$. Write $V_{n} = \alpha_{n} X(p_{n}) = \alpha_{n} \left(x_{n}(y_{n} + z_{n}), y_{n}(x_{n}+z_{n}), z_{n}(x_{n}+y_{n})\right)$ for some $\alpha_{n} \in \mathbb{C}$. We get
    \begin{equation}\label{E:Eq3}
    \alpha_{n} x_{n}(y_{n}+z_{n}) \to v_{1}, \ \alpha_{n} y_{n}(x_{n}+z_{n}) \to 0, \ \& \ \alpha_{n} z_{n}(x_{n}+y_{n}) \to 0.
    \end{equation}
    Since $(x_{n} + z_{n}), (x_{n} + y_{n}) \to x \neq 0$, we get from \ref{E:Eq3} that $\alpha_{n} y_{n}, \alpha_{n} z_{n} \to 0$. Therefore 
    \[
    \alpha_{n} x_{n} (y_{n} + z_{n}) = x_{n} (\alpha_{n} y_{n} + \alpha_{n} z_{n}) \to x(0+0) = 0 = v_{1}.
    \]
    Thus $V = (0,0,0)$ and $\mathcal{F}$ is transversal type at $p$. Exactly similar calculations will prove that $\mathcal{F}$ is transversal type at each $p \in E \sm \{0\}$. 

    \noindent For $p = (0,0,0) \in E$, note that 
    \[
    C_{0}E = \text{Span}_{\mathbb{C}}\{(1,0,0)\} \cup \text{Span}_{\mathbb{C}}\{(0,1,0)\} \cup \text{Span}_{\mathbb{C}}\{(0,0,1)\}.
    \]
    Fix $V= (1,0,0) \in C_{0}E$. Consider $p_{n} = \left(\frac{1}{n}, \frac{-1}{n},\frac{-1}{n}\right) \in M \sm E$, and $\alpha_{n} = \frac{-n^2}{2}$. We have $p_{n} \to 0$ and
    \[
    V_{n} = \alpha_{n} X(p_{n}) = \frac{-n^2}{2} \left(\frac{1}{n}\left(\frac{-1}{n}+ \frac{-1}{n}\right), \frac{-1}{n}\left(\frac{1}{n}+ \frac{-1}{n}\right), \frac{-1}{n}\left(\frac{1}{n}+ \frac{-1}{n}\right)\right)  = (1,0,0) = V.
    \]
    We get $(p_{n}, V_{n}) \to (0,V)$ for $V= (1,0,0) \neq 0$. Therefore, $\mathcal{F}$ is not transversal type at $0 \in E$. But Theorem \ref{T:Cont-Extn} tells us that there is still a continuous extension of $\eta$ on $M$.
\end{exam}

%%%%%%%%%%%%%%%%%%%%%%%%%%%%%%%%%%%%%%%%%%%%

%\noindent \textbf{Acknowledgement:} The author would like to thank Viêt-Anh Nguyên and Kaushal Verma for giving their valuable feedback on the preprint version of this article.
%%%%%%%%%%%%%%%%%%%%%%%%%%%%%%%%%%%%%%%%%%%%%

\section{Proof of Theorem \ref{T:Removable Singular sets}}
\noindent\begin{enumerate}
    \item It is immediate, according to the definitions \ref{D:Al} and \ref{D:B_l}.
    \medskip
    
    \item Fix $1 \le l \le k$, and let $p \in A_{l}$. By definition \ref{D:Al}, there exists an $\mathcal{F}-$invariant submanifold $\Sigma$ of dimension $l+1$ such that $\Sigma \cap E$ is of dimension $l$. Therefore, $\text{dim}_{p}E \ge l$ which gives $p \notin E_{m}$ for $m < l$. The fact that $p \in E$ and $B_{l} \subset A_{l}$ gives us $B_{l}, A_{l} \subset \cup_{m=l}^{k}E_{m}$.
    \medskip

    \item Suppose $p \in A_{l}$ for some $2 \le l \le k$. Let $\Sigma$ be the submanifold of dimension $l+1$ such that $p \in A_{l,\Sigma}$. Since $p$ is a regular point of the foliation $\hat{\mathcal{F}}\vert_{\Sigma}$, there exists an open set $U \subset \Sigma$ containing $p$ and an injective holomorphic map $\phi : U \to \mathbb{C}^{l+1}$ with the property that $\phi(p) = 0$ and the foliation $\phi_{*}(\hat{\mathcal{F}}\vert_{\Sigma})$ is given by the vector field $X \equiv \frac{\partial}{\partial z_{1}}$.

    For $A = (a_{2}, a_{3}, \ldots, a_{l+1}) \in \mathbb{C}^{l} \sm \{0\}$, define $\Sigma_{A} := \{z \in \mathbb{C}^{l+1} : \sum_{j=2}^{l+1}{a_{j} z_{j}} = 0\}$. It is clear that $\Sigma_{A}$ is an $l$-dimensional submanifold of $\phi(U)$, and is invariant under $\phi_{*}(\hat{\mathcal{F}}\vert_{\Sigma})$. Since $\phi(E \cap \Sigma)$ is of dimension $l$ and
    \[
    \phi(U) \subset \cup_{A \in \mathbb{C}^{l} \sm \{0\}} {\Sigma_{A}},
    \]
    there exists some $A \in \mathbb{C}^{l} \sm \{0\}$ such that $\Sigma_{A} \not\subset \phi(E \cap \Sigma)$. Now using the fact that $0 \in \Sigma_{A} \cap \phi(E \cap \Sigma)$, we get that $\Sigma_{A} \cap \phi(E \cap \Sigma)$ is of dimension $l-1$. 

    \noindent Define $\tilde{\Sigma} := \phi^{-1}(\Sigma_{A})$. Note that $\tilde{\Sigma}$ is a submanifold of dimension $l$, $E \cap \tilde{\Sigma}$ is an $(l-1)-$dimensional analytic subset, and $p$ is regular point of the foliation $\hat{\mathcal{F}}\vert_{\tilde{\Sigma}}$. Therefore $p \in A_{l-1, \tilde{\Sigma}} \subset A_{l-1}$. Thus $A_{l} \subset A_{l-1}$.
    \medskip

    \item Let $p \in B_{l}$ for some $2 \le l \le k$. Let $\Sigma$ be the submanifold of dimension $l+1$ such that $p \in B_{l,\Sigma}$ and the leaf $\hat{L}_{p}$ of the foliation $\hat{\mathcal{F}}\vert_{\Sigma}$ satisfies $\hat{L}_{p} \subset E \cap \Sigma$. Similar to the previous step, we consider an open set $U \subset \Sigma$ containing $p$ and an injective holomorphic map $\phi : U \to \mathbb{C}^{l+1}$ such that $\phi(p) = 0$, and the foliation $\phi_{*}(\hat{\mathcal{F}}\vert_{\Sigma})$ is induced by the vector field $X \equiv \frac{\partial}{\partial z_{1}}$.

    In these coordinates, the leaf $L_{0}$ of the foliation $\phi_{*}(\hat{\mathcal{F}}\vert_{\Sigma})$ passing through $0 \in \mathbb{C}^{l+1}$ is given by 
    \[
    L_{0} = \{z \in \mathbb{C}^{l+1} : z_{2} = z_{3} = \ldots = z_{l+1} = 0\}.
    \]
    Using the fact that $p \in B_{l, \Sigma}$, we get $L_{0} \subset \phi(E \cap \Sigma)$. Again, there exist some $A \in \mathbb{C}^{l} \sm \{0\}$ such that $\Sigma_{A}$ is an $l-$dimensional $\phi_{*}(\hat{\mathcal{F}}\vert_{\Sigma})-$invariant submanifold of $\phi(U)$ with the property that $\Sigma_{A} \cap \phi(E \cap \Sigma)$ is of dimension $l-1$. Also since $0 \in \Sigma_{A}$, the corresponding leaf $L_{0} \subset \Sigma_{A} \cap \phi(E \cap \Sigma)$. 

    \noindent Define $\tilde{\Sigma} := \phi^{-1}(\Sigma_{A})$. Note that $\tilde{\Sigma}$ is a submanifold of dimension $l$, $E \cap \tilde{\Sigma}$ is an $(l-1)-$dimensional analytic subset, and $p$ is regular point of the foliation $\hat{\mathcal{F}}\vert_{\tilde{\Sigma}}$ such that the corresponding leaf $L_{p} \subset \tilde{\Sigma} \cap E$. Therefore $p \in B_{l-1, \tilde{\Sigma}} \subset B_{l-1}$. Thus $B_{l} \subset B_{l-1}$.
    \medskip

    \item Let $p \in A_{l} \sm B_{l}$ for some $1 \le l \le k$. By definition, $p \in A_{l,\Sigma}$ for some $(l+1)-$dimensional $\mathcal{F}-$invariant submanifold $\Sigma$, such that $\Sigma \cap E$ is of dimension $l$ and $p$ is a removable singularity of the foliation $\hat{\mathcal{F}}\vert_{\Sigma}$. Let $\hat{L}_{p}$ be the leaf of $\hat{\mathcal{F}}\vert_{\Sigma}$ passing through $p$. Since $p \notin B_{l}$, therefore $\hat{L}_{p} \not\subset E \cap \Sigma$. Consider a neighborhood $U \subset \Sigma$ of $p$ small enough such that the component of $L_{p} \cap U$ which contains $p$ is an analytic curve. This is possible since $p$ is a regular point of the foliation $\hat{\mathcal{F}}\vert_{\Sigma}$. We denote the component by $\mathcal{C}$. 

    Next we claim that we can choose U to be small enough so that 
    \[
    \mathcal{C} \cap (E \cap \Sigma) = \{p\}.
    \]
    Suppose not, then there exist a sequence of points $p_{n} \in \mathcal{C} \sm \{p\}$ such that $p_{n} \to p$ and $p_{n} \in \mathcal{C} \cap (E \cap \Sigma)$ for all $n \in \mathbb{N}$. By the local structure of holomorphic foliation, we can choose $U$ small enough such that $\mathcal{C}$ is biholomorphic to $\mathbb{D}$. Let $\alpha : \mathbb{D} \to \mathcal{C} \subset L_{p}$ be the biholomorphic map such that $\alpha(0) = p$. The corresponding sequence $q_{n} := \alpha^{-1}(p_{n})$ satisfies $q_{n} \to 0$. Also let $f_{1}, f_{2}, \ldots, f_{r} \in \mathcal{O}(U)$ defines the analytic set $E\cap \Sigma$ near $p$, that is $E \cap U = \{z \in U : f_{1}(z) = f_{2}(z) = \ldots = f_{r}(z) = 0\}$. For $1 \le j \le r$, consider the map $f_{j} \circ \alpha : \mathbb{D} \to \mathbb{C}$. Since $p_{n} = \alpha(q_{n}) \in E \cap U$, for all $n \in \mathbb{N}$, therefore $f_{j} \circ \alpha (q_{n}) = 0$. By Identity theorem, $f_{j} \circ \alpha \equiv 0$, for all $1 \le j \le r$. This tells us that $f_{j}(z) = 0$ for all $z \in \mathcal{C}$ and $1 \le j \le r$. Therefore $\alpha(\mathbb{D}) = \mathcal{C} \subset E \cap U \subset E \cap \Sigma$, which is a contradiction.

    \noindent Thus, we can choose $U$ to be small enough so that $\mathcal{C}\sm \{p\}$ is contained in a leaf of the foliation $\mathcal {F}$, and so $p \in A_{0}$.

\end{enumerate} 

%%%%%%%%%%%%%%%%%%%%%%%%%%%%%%%%%%%%%%%%%%%%%%%

\section{Proof of Theorem \ref{T: Regularity of metric}}
Let $\{\alpha_{n}\}_{n \ge 1} \subset \mathcal{U}$ be a sequence of uniformizations of leaves of $\mathcal{F}$ such that $\alpha_{n} \to \alpha$ on compact subsets of $\mathbb{D}$, where $\alpha : \mathbb{D} \to M$ is holomorphic. Suppose $p = \alpha(0) \in E$. The following result (Theorem 1.7 in \cite{GV2}) tells us that it suffices to prove $\alpha(\mathbb{D}) \subset E$, to achieve the continuity of $\eta$ on $M \sm E$. 

\begin{res}\label{R:Result1}
    Let $\mathcal{F}$ be a singular hyperbolic foliation on a complex manifold $M$, with singular set $E \subset M$. Also, let $g$ be a given hermitian metric on $M$, and $\eta$ be the modulus of uniformization map of $\mathcal{F}$. Suppose that $\mathcal{U}$ is NCP. Then the following are equivalent:

\begin{enumerate}
\item $\eta$ is continuous in $M \setminus E$.
\item For any sequence $\{\alpha_n\}_{n \ge 1}$ in $\mathcal{U}$, which converges on compact subsets of $\mbb{D}$ to some $\alpha : \mbb{D} \rightarrow M$, and $p = \alpha(0) \notin E$, then $\alpha(\mbb{D}) \subset L_{p}$, where $L_p$ is the leaf of $\mathcal{F}$ passing through $p$. 
\item For any sequence $\{\alpha_n\}_{n \ge 1}$ in $\mathcal{U}$, which converges on compact subsets of $\mbb{D}$ to some $\alpha : \mbb{D} \rightarrow M$, and $p = \alpha(0) \notin E$, then $\alpha$ is a uniformization of $L_p$.
\item For any sequence $\{\alpha_n\}_{n \ge 1}$ in $\mathcal{U}$, which converges on compact subsets of $\mbb{D}$ to some $\alpha : \mbb{D} \rightarrow M$, and $p = \alpha(0) \in E$, then $\alpha(\mbb{D}) \subset E$.
\end{enumerate}
\end{res}

Next, we use the analyticity of $E$ and holomorphicity of the map $\alpha$ to prove that the condition $\alpha(\mathbb{D}) \subset E$ is equivalent to a much weaker condition, that is, $\alpha^{-1}(E) \subset \mathbb{D}$ has an accumulation point in $\mathbb{D}$. 

\begin{lem}\label{L:Lemma1}
Let $\mathcal{F}$ be a singular holomorphic foliation on a complex manifold $M$, with singular set $E \subset M$. Let $\alpha : \mathbb{D} \to M$ be a holomorphic map. Then $\alpha(\mathbb{D}) \subset E$ if and only if $\alpha^{-1}(E) \subset \mathbb{D} $ has an accumulation point in $\mathbb{D}$.
\end{lem}

\begin{proof}
    If $\alpha(\mathbb{D}) \subset E$, it is clear that $\alpha^{-1}(E) = \mathbb{D} \subset \mathbb{D}$ has every point in $\mathbb{D}$ as an accumulation point. So, the only non-trivial part is the converse. Suppose $\alpha^{-1}(E) \subset \mathbb{D}$ has an accumulation point $z_{0} \in \mathbb{D}$. Since $E\subset M$ is closed and $\alpha$ is holomorphic, therefore $\alpha^{-1}(E)$ is a closed subset of $\mathbb{D}$. Therefore $z_{0} \in \alpha^{-1}(E)$. Let $z_{n} \in \alpha^{-1}(E) \subset \mathbb{D}$ be a sequence of points such that $z_{n} \to z_{0}$. Define $p_{0} := \alpha(z_{0}) \in E$ and $p_{n} := \alpha(z_{n}) \in E$. Since $E$ is an analytic set, there exists a neighborhood $U \subset M$ of $p_{0}$ and holomorphic functions $f_{1}, f_{2}, \ldots, f_{s} \in \mathcal{O}(U)$ for some $s \ge 2$, such that $E \cap U = \{p \in U : f_{1}(p) = f_{2}(p) = \ldots = f_{s}(p) = 0\}$. Choose $r >0$ be small enough that $\alpha(D(z_{0}, r)) \subset U$. Consider the holomorphic functions $f_{1} \circ \alpha\vert_{D(z_{0},r)}, f_{2} \circ \alpha\vert_{D(z_{0},r)}, \ldots, f_{s} \circ \alpha\vert_{D(z_{0},r)} \in \mathcal{O}(D(z_{0},r))$. Since $z_{n} \to z_{0}$, there exist $N_{1} \ge 1$ such that $z_{n} \in D(z_{0},r)$ for all $n \ge N_{1}$. Note that for $1 \le i \le s$, and for all $n \ge N_{1}$ 
    \[
    (f_{i} \circ \alpha\vert_{D(z_{0},r)})(z_{n}) = (f_{i} \circ \alpha\vert_{D(z_{0},r)})(z_{0}) = 0.
    \]
    By Identity theorem, we get $f_{i} \circ \alpha\vert_{D(z_{0},r)} \equiv 0$ for all $1 \le i \le s$. Therefore, $D(z_{0},r) \subset \alpha^{-1}(E \cap U) \subset \alpha^{-1}(E)$. Let $\hat{D} \subset \mathbb{D}$ be the maximal closed connected subset such that $z_{0} \in \hat{D}$ and $\hat{D} \subset \alpha^{-1}(E)$. We already proved that $\hat{D} \neq \{z_{0}\}$, in fact, the interior $\text{int}(\hat{D}) \neq \emptyset$. We claim $\hat{D} = \mathbb{D}$. Suppose not, then there exists an accumulation point $\hat{z} \in \mathbb{D}$ of the set $\hat{D}$ which satisfies $\hat{z} \in \partial{\hat{D}}$. By the closedness of the set $\hat{D}$, $\hat{z} \in \hat{D}$. Using the above argument, we observe that there exists $r>0$ such that $\alpha(D(\hat{z},r)) \subset E$. Consider the set $\tilde{D} = \hat{D} \cup \overline{D(\hat{z},\frac{r}{2})} \subset \mathbb{D}$, and note that $\tilde{D}$ is a closed connected subset of $\mathbb{D}$ such that $z_{0} \in \tilde{D}$ and $\tilde{D} \subset \alpha^{-1}(E)$. By the maximality of $\hat{D}$, we get $\hat{D} = \tilde{D}$. But this contradicts the fact that $\hat{z} \in \partial{\hat{D}}$, since $\hat{z} \in \text{int}(\tilde{D})$. Therefore, $\hat{D} = \mathbb{D}$, and thus $\alpha(\mathbb{D}) \subset E$.
\end{proof}

\noindent The next proposition along with Result \ref{R:Result1} and Lemma \ref{L:Lemma1} will conclude the proof.

\begin{prop}\label{P:Prop1}
   Let $\mathcal{F}$ be a singular hyperbolic foliation on a complex manifold $M$ and singular set $E$. Suppose that $\mathcal{U}$ is NCP, $A_{0} \subset B_{1}$, and for each $p \in A_{0}$, there exist $(N-1 - l_{p})$ number of local $\mathcal{F}$-invariant hypersurfaces $\Sigma_{1,p}, \Sigma_{2,p}, \ldots, \Sigma_{N-1 -l_{p},p}$ such that $\Sigma_{p} = \cap_{i =1}^{N-1-l_{p}} \Sigma_{i,p}$ is a $(l_{p} +1)-$dimensional submanifold and $p \in B_{l_{p}, \Sigma_{p}}$. If $\{\alpha_{n}\}_{n \ge 1} \subset \mathcal{U}$ is a sequence of uniformizations which converges on compact subsets of $\mbb{D}$ to some $\alpha : \mbb{D} \rightarrow M$, and $p = \alpha(0) \in E$, then there exists $r >0$ such that $\alpha(D(0,r)) \subset E$.
\end{prop}

\begin{proof}
   Since $p = \alpha(0) \in E$, there are only two possibilities, either $p \notin A_{0}$ or $p \in A_{0}$. 
   
   \noindent (a) Suppose $p \notin A_{0}$. If $0 \in \alpha^{-1}(p) \subset \mathbb{D}$ is an isolated point of the set $\alpha^{-1}(E)$, then there exists $r_{1} > 0$ such that $\alpha(D(0,r_{1})\sm \{0\}) \subset M \sm E$. By Lemma 1.6 in \cite{GV2} (page 96), we get $\alpha(D(0,r_{1})\sm \{0\}) \subset L$, where $L$ is a leaf of the foliation $\mathcal{F}$. Thus, $\mathcal{C} = \alpha(D(0,r_{1}))$ is a local separatrix through $p \in E$, and $\mathcal{C} \sm \{p\} \subset M \sm E$, contradicting the fact that $p \notin A_{0}$. Therefore, $0 \in \mathbb{D}$ is an accumulation point of the set $\alpha^{-1}(E)$. By the arguments used in the proof of the Lemma \ref{L:Lemma1}, there exists $r>0$ such that $\alpha(D(0,r)) \subset E$.

   \noindent (b) Now suppose $p \in A_{0}$. By hypothesis, $p \in B_{l_{p}, \Sigma_{p}}$, where $p \in \Sigma_{p} = \cap_{i =1}^{N-1-l_{p}} \Sigma_{i,p}$ and $\Sigma_{i, p}$'s are analytic local $\mathcal{F}-$invariant hypersurfaces, for $1 \le i \le N-1-l_{p}$. Choose a small enough neighborhood $U$ of $p$ such that for $1 \le i \le N-1-l_{p}$,
   \[
   \Sigma_{i,p} = \{f_{i} = 0\},
   \]
   where $f_{i} \in \mathcal{O}(U)$. Since $p \in U$ and $\alpha_{n} \to \alpha$ on the compacts subsets of $\mathbb{D}$, there exists $r_{2} >0$ and $N_{1} \ge 1$ such that $\alpha_{n}(D(0,r_{2})), \alpha(D(0,r_{2})) \subset U$ for all $n \ge N_{1}$. We claim that $\alpha(D(0,r_{2})) \subset \Sigma_{i,p}$ for all $1 \le i \le N-1-l_{p}$. Fix $1 \le i \le N-1-l_{p}$. Consider for $n \ge N_{1}$, the set $\alpha_{n}(D(0,r_{2})) \subset U$. Since $\Sigma_{i,p}$ is an $\mathcal{F}-$invariant hypersurface and $\alpha_{n} \in \mathcal{U}$, therefore either $\alpha_{n}(D(0,r_{2})) \subset \Sigma_{i,p}$ or $\alpha_{n}(D(0,r_{2})) \cap \Sigma_{i,p} = \emptyset$. 
   
   \noindent If there is a subsequence $\{\alpha_{n_{k}}\}_{k \ge 1}$ of $\{\alpha_{n}\}_{n \ge N_{1}}$ such that for all $k \ge 1$, 
   \[
   \alpha_{n_{k}}(D(0,r_{2})) \subset \Sigma_{i,p},
   \]
   that is, $f_{i} \circ \alpha_{n_{k}}\vert_{D(0,r_{2})} \equiv 0$, then $f_{i} \circ \alpha_{n_{k}}\vert_{D(0,r_{2})} \to f_{i} \circ \alpha\vert_{D(0,r_{2})} \equiv 0$. Therefore, we get $\alpha(D(0,r_{2})) \subset \Sigma_{i,p}$. 
   
   \noindent If there does not exists such subsequence, then there exists $N_{2} \ge 1$ such that for all $n \ge N_{2}$,
   \[
   \alpha_{n}(D(0,r_{2})) \cap \Sigma_{i,p} = \emptyset,
   \]
   that is $f_{i} \circ \alpha_{n}\vert_{D(0,r_{2})} (z) \neq 0$ for all $z \in D(0,r_{2})$. But $f_{i} \circ \alpha_{n}\vert_{D(0,r_{2})} \to f_{i} \circ \alpha\vert_{D(0,r_{2})}$ uniformly on $D(0,r_{2})$, and $f_{i} \circ \alpha\vert_{D(0,r_{2})}(0) = f_{i}(p) = 0$. Therefore, Hurwitz's theorem tells us that 
   \[
   f_{i} \circ \alpha\vert_{D(0,r_{2})} \equiv 0,
   \]
   that is $\alpha(D(0,r_{2})) \subset \Sigma_{i,p}$. So we have our claim that $\alpha(D(0,r_{2})) \subset \Sigma_{i,p}$, for all $1 \le i \le N-1-l_{p}$. Therefore, $\alpha(D(0,r_{2})) \subset \Sigma_{p} = \cap_{i=1}^{N-1-l_{p}}{\Sigma_{i,p}}$.

   \medskip

   \noindent Recall that $\Sigma_{p}$ is a $(l_{p}+1)-$dimensional $\mathcal{F}-$invariant submanifold, and $p \in B_{l_{p}, \Sigma_{p}}$. We now consider the saturated foliation $\hat{\mathcal{F}}\vert_{\Sigma_{p}}$ of $\Sigma_{p}$. By definition of the set $B_{l_{p}, \Sigma_{p}}$, the leaf $\hat{L}_{p} \subset E \cap \Sigma_{p} \subset E$. Suppose if possible $0 \in \alpha^{-1}(E) \subset \mathbb{D}$ is an isolated point of the set $\alpha^{-1}(E)$. Then there exists $r_{3} > 0$ such that $\alpha(D(0,r_{3}) \sm \{0\}) \subset \Sigma_{p} \sm E$. Again by Lemma 1.6 (in \cite{GV2}), 
   \[
   \alpha(D(0,r_{3}) \sm \{0\}) \subset L,
   \]
   where $L$ is a leaf of $\mathcal{F}$. Therefore, $\alpha(D(0,r_{3}))$ is an invariant curve of the foliation $\mathcal{F}\vert_{\Sigma_{p}}$, which in turn tells us that it is also invariant under the saturated foliation $\hat{\mathcal{F}}\vert_{\Sigma_{p}}$. Thus, $\alpha(D(0,r_{3})) \subset \hat{L}_{p}$. Since $\hat{L}_{p} \subset E$, we have 
   \[
   \alpha(D(0,r_{3})) \subset E.
   \]
    This is a contradiction to the assumption that $0 \in \alpha^{-1}(E)$ is an isolated point. Thus, $0 \in \alpha^{-1}(E)$ is an accumulation point and therefore there exists $r >0$ such that $\alpha(D(0,r)) \subset E$.
\end{proof}

\medskip
%%%%%%%%%%%%%%%%%%%%%%%%%%%%%%%%%%%%%%%%%%%%%%%%

\section{Proof of Theorem \ref{T:Mix Regularity}}
Following the proof of Theorem \ref{T: Regularity of metric}, it suffices to prove that if $\{\alpha_{n}\}_{n \ge 1} \subset \mathcal{U}$ is a sequence of uniformizations of $\mathcal{F}$ such that $\alpha_{n} \to \alpha$ on compact subsets of $\mathbb{D}$, where $\alpha : \mathbb{D} \to M$ is holomorphic and $p = \alpha(0) \in E$, then $\alpha(\mathbb{D}) \subset E$.

\noindent If $p \notin A_{0}$ or $p \in E_{L}$, then we already observed that $\alpha(\mathbb{D}) \subset E$. So, the only case remaining is when $p \in E_{M}$. By definition of the set $E_{M}$, there exists a neighborhood $U$ of $p$, and $(N-k)$ number of local $\mathcal{F}-$invariant hypersurfaces $\tilde{\Sigma}_{1,p}, \tilde{\Sigma}_{2,p}, \ldots, \tilde{\Sigma}_{N-k,p} \subset U$ such that 
    \[
    p \in \cap_{i=1}^{N-k}{\tilde{\Sigma}_{i,p}} \subset E,
    \]
    where $k = \text{dim}_{p}E$. As we noted in the proof of Proposition \ref{P:Prop1}, there exists $r >0$ such that $\alpha(D(0,r)) \subset \tilde{\Sigma}_{i,p}$, for all $1 \le i \le N-k$. Therefore, 
    \[
    \alpha(D(0,r)) \subset \cap_{i=1}^{N-k}{\tilde{\Sigma}_{i,p}} \subset E \cap U \subset E.
    \]
    Using Lemma \ref{L:Lemma1}, we get $\alpha(\mathbb{D}) \subset E$, and we have our result.

\medskip

%%%%%%%%%%%%%%%%%%%%%%%%%%%%%%%%%%%%%%%%%%%%

\section{Proof of Theorem \ref{T:Cont-Extn}}
Consider the map $\tilde{\eta} : M \to [0,\infty)$ given by
\[
\tilde{\eta}(p) := 
\begin{cases}
       \eta(p), &\quad\text{if} \quad p \in M \sm E\\
       0, &\quad\text{if} \quad p \in E.\\ 
     \end{cases}
\]
We will prove that $\tilde{\eta}$ is continuous on $M$. Since $\eta$ is continuous on $M \sm E$, it suffices to check the continuity of $\tilde{\eta}$ on the set $E$. 

\noindent Let $p \in E$ be an arbitrary singular point of the foliation $\mathcal{F}$. If $\mathcal{F}$ is transversal type at $p$, then Theorem 1.8 in \cite{GV2} (see page 96) tells us that $\tilde{\eta}$ is continuous at $p$. So, we assume that $\mathcal{F}$ is not transversal type at $p \in E$. Suppose, if possible, that $\tilde{\eta}$ is not continuous at $p$. Then there exist $\epsilon >0$, and a sequence $\{p_{n}\}_{n \ge 1} \subset M \sm E$ of non-singular points of $\mathcal{F}$ such that $p_{n} \to p$, and 
\begin{equation}\label{E:Equation5.1}
\tilde{\eta}(p_{n}) = \eta(p_{n}) > \epsilon,
\end{equation}
for all $n \ge 1$. For $n \ge 1$, let $\alpha_{n} : \mathbb{D} \to L_{p_{n}}$ be a uniformization of the leaf $L_{p_{n}}$ such that $\alpha_{n}(0) = p_{n}$, that is $\alpha_{n} \in \mathcal{U}$. Since $\alpha_{n}(0) = p_{n} \to p$, and $\mathcal{U}$ is NCP, therefore upto a subsequence we can assume that $\alpha_{n} \to \alpha$ uniformly on the compact subsets of $\mathbb{D}$, where $\alpha: \mathbb{D} \to M$ is some holomorphic map satisfying $\alpha(0) = p \in E$. Using Result \ref{R:Result1}, we get $\alpha(\mathbb{D}) \subset E$. We claim that $\alpha$ is a constant map, that is, $\alpha \equiv p$. Assuming the above claim, we get $\alpha'(0) = 0$. Since $\alpha_{n} \to \alpha$ uniformly on compact subsets of $\mathbb{D}$, therefore $\alpha_{n}'(z) \to \alpha'(z)$ for all $z \in \mathbb{D}$. In particular,
\[
\tilde{\eta}(p_{n}) = \eta(p_{n}) = \vert \alpha_{n}'(0)\vert_{g} \to \vert \alpha'(0)\vert_{g} = 0.
\]
So, there exists $N_{1} \ge 1$ such that $\tilde{\eta}(p_{n}) < \epsilon$, for all $n \ge N_{1}$. But this contradicts \ref{E:Equation5.1}, and thus $\tilde{\eta}$ is continuous at $p \in E$. 

\noindent So, it suffices to prove our claim that $\alpha \equiv p$. Suppose not, then $\alpha^{-1}(\{p\}) \subset \mathbb{D}$ cannot have accumulation points in $\mathbb{D}$, and $\alpha' \not\equiv 0$. Therefore, there exists $r >0$ such that for all $z \in D(0,r) \sm \{0\}$, we have $\alpha(z) \in E \sm \{p\}$, and $\alpha'(z) \neq 0$. Since, the set $E \sm \{q \in E : \mathcal{F} \ \text{is transversal type at} \ q\}$ is a discrete subset of $M$, we can choose $r>0$ small enough such that $\mathcal{F}$ is transversal type at each $q \in \alpha(D(0,r)\sm \{0\})$. Now, observe that for all $z \in \mathbb{D}$, 
\begin{equation}\label{E:Equation5.2}
\alpha'(z) \in C_{\alpha(z)}{E}.
\end{equation}
This is due to the fact that $\alpha(\mathbb{D}) \subset E$. Also, since $\alpha_{n}'(z) \in T_{\alpha_{n}(z)}{\mathcal{F}}$ for all $z \in \mathbb{D}$, and $\alpha_{n}'(z) \to \alpha'(z)$, therefore 
\begin{equation}\label{E:Equation5.3}
\alpha'(z) \in C_{\alpha(z)}{\mathcal{F}},
\end{equation}
for all $z \in \mathbb{D}$. Combining equations \ref{E:Equation5.2} and \ref{E:Equation5.3}, we get $\alpha'(z) \in C_{\alpha(z)}{\mathcal{F}} \cap C_{\alpha(z)}{E}$, for all $z \in \mathbb{D}$. Note that for all $z \in D(0,r) \sm \{0\}$, 
\[
0 \neq \alpha'(z) \in C_{\alpha(z)}{\mathcal{F}} \cap C_{\alpha(z)}{E}.
\]
Therefore, $\mathcal{F}$ cannot be transversal type at $\alpha(z) \in E$, for all $z \in D(0,r) \sm \{0\}$. This is a contradiction. Therefore $\alpha \equiv p$, and we have our claim.

\medskip

%%%%%%%%%%%%%%%%%%%%%%%%%%%%%%%%%%%%%%%%%%%%%%%

\section{Remarks}

\noindent \textbf{Remark 1:} One can look at \cite{Car-San} and \cite{Salas}, where sufficient conditions have been provided for the existence of local invariant submanifolds passing through singular points of a foliation. Combined with a suitable condition on the dimension of the singular set, one can establish $A_{1} \neq \emptyset$ for such foliations. Furthermore, Salas \cite{Salas} (see Theorem 2.6, page 736) demonstrated that under certain conditions on the singular set, almost every point $p \in E_{N-2}$ belongs to $A_{0} \cap A_{N-2}$, implying that $E_{N-2} \cap A_{0} \cap A_{N-2}$ has full measure in $E_{N-2}$.  Investigating these conditions further could also provide insight into the structure of the sets $B_{l}$.

\medskip

\begin{exam}\label{E:Example6.1}
    Let $M = \mathbb{D}^3 \subset \mathbb{C}^3$ and $\mathcal{F}$ be the singular hyperbolic foliation on $M$ induced by the holomorphic vector field
    \[
    X(x,y,z) = (y^{2} -x^3)\frac{\partial}{\partial x} + (y^2 - x^3)\frac{\partial}{\partial y} + x \frac{\partial}{\partial z}.
    \]
    It is easy to see that $\text{sing}(\mathcal{F}) = E = \{(0,0,z) : z \in \mathbb{D}\}$. Consider the hypersurface $\Sigma := \{(x,y,z) \in M : y^{2} = x^{3}\}$. Note that for $p = (x,y,z) \in \Sigma$, 
    \[
    X(p) = 0\frac{\partial}{\partial x} + 0\frac{\partial}{\partial y} + x \frac{\partial}{\partial z} = x \frac{\partial}{\partial z} \in T_{p}{\Sigma}.
    \]
    Therefore $\cup_{p \in \Sigma}{T_{p}{\mathcal{F}}} = T\mathcal{F}\vert_{\Sigma} \subset \cup_{q \in \Sigma}{T_{q}{\Sigma}} $, and thus $\Sigma$ is an $\mathcal{F}-$invariant hypersurface. One can easily check that if $p_{0} = (x_{0},y_{0},z_{0}) \in \Sigma \sm E$, that is $x_{0} \neq 0$, then the corresponding leaf $L_{p_{0}}$ is given by
    \[
    L_{p_{0}} = \{(x_{0}, y_{0}, z) : z \in \mathbb{D}\}.
    \]
    Let $\{\alpha_{n}\}_{n \ge 1} \subset \mathcal{U}$ be a sequence of uniformizations of leaves of $\mathcal{F}$ such that $\alpha_{n} \to \alpha$ uniformly on compact subsets of $\mathbb{D}$, where $\alpha : \mathbb{D} \to M$ is holomorphic and $\alpha(0) = \hat{p} = (0,0,\hat{z}) \in E$. Since $\hat{p} \in E \subset \Sigma$, we can use the argument in the proof of Proposition \ref{P:Prop1} to prove that $\alpha(D(0,r)) \subset \Sigma$ for some $r>0$. 

    \noindent If the set $\alpha^{-1}(E) \subset \mathbb{D}$ has an accumulation point in $\mathbb{D}$, then we can choose $r>0$ small enough so that $\alpha(D(0,r)) \subset E$. This will further imply that $\alpha(\mathbb{D}) \subset E$. Using Result \ref{R:Result1}, we get the continuity of the map $\eta$ on $M \sm E$.

    \noindent Suppose if possible, that the set $\alpha^{-1}(E) \subset \mathbb{D}$ has no accumulation point in $\mathbb{D}$. Therefore, we can choose $r>0$ small enough so that $\alpha(D(0,r)\sm\{0\}) \subset M \sm E$. Lemma 1.6 (in \cite{GV2}) tells us that $\alpha(D(0,r)\sm \{0\}) \subset L$, for some leaf $L$ of the foliation $\mathcal{F}$. Let $\xi \in D(0,r) \sm \{0\}$, and consider $\alpha(\xi) = (x_{1}, y_{1}, z_{1}) \in \Sigma \sm E$. Clearly $x_{1} \neq 0$, and $L_{\alpha(\xi)} = \{(x_{1}, y_{1}, z) : z \in \mathbb{D}\}$. Since $\alpha(D(0,r)\sm \{0\}) \subset L$, therefore $L = L_{\alpha(\xi)}$. Thus we get
    \[
    \alpha(D(0,r) \sm \{0\}) \subset L_{\alpha(\xi)} = \{(x_{1}, y_{1}, z) : z \in \mathbb{D}\},
    \]
    where $x_{1} \neq 0$. This is a contradiction to the fact that $\alpha\vert_{D(0,r)} : D(0,r) \to M$ is holomorphic and $\alpha(0) = (0,0,\hat{z}) \in E$. Therefore, the set $\alpha^{-1}(E) \subset \mathbb{D}$ must have an accumulation point in $\mathbb{D}$. Thus, the modulus of uniformization map $\eta$ is continuous on $M \sm E$.
\end{exam}

\noindent \textbf{Remark 2:} 
%We note that there is a difference in the way that the set $A_{0}$ and $A_{l}$, for $l \ge 1$, are defined respectively in \ref{D:A0} and \ref{D:Al}. For $A_{0}$, we ask for the existence of some invariant analytic curve, which can possibly have singularities (as an analytic variety) at the corresponding singular points of the foliation. But for $A_{l}$, with $l \ge 1$, we ask for some invariant complex submanifold. The reason behind this is to make sense of the subfoliation and the corresponding removable singular points. But the previous example \ref{E:Example6.1} hints us that there may be some way of generalizing these notions to more general objects, for example, hypersurfaces. This could help us to achieve the conclusions of Theorem \ref{T: Regularity of metric}, Corollary \ref{C: Regularity of metric}, and Theorem \ref{T:Mix Regularity} for a bigger class of foliations.
We observe a fundamental distinction in the definitions of the sets $A_{0}$, and $A_{l}$ for $l\ge1$, as given in the Definitions \ref{D:A0} and \ref{D:Al}, respectively. In the case of $A_{0}$, the condition requires the existence of an invariant analytic curve, which may itself exhibit singularities as an analytic variety at the singular points of the foliation. However, for $A_{l}$ with $l\ge 1$, we impose the stricter requirement of an invariant complex submanifold. This distinction is necessary to properly define the subfoliation and the corresponding removable singular points. Nevertheless, Example \ref{E:Example6.1} suggests the possibility of extending these notions to more general objects, such as hypersurfaces. Such a generalization could potentially allow us to establish the conclusions of Theorem \ref{T: Regularity of metric}, Corollary \ref{C: Regularity of metric}, and Theorem \ref{T:Mix Regularity} for a broader class of foliations.

\medskip

\noindent \textbf{Remark 3:} In Theorem \ref{T:Cont-Extn}, we assume that the set $E \sm \{p \in E: \mathcal{F} \ \text{is transversal type at} \ p\}$ is discrete in order to establish the existence of a continuous extension of the map $\eta$ to the entirety of $M$. If we replace this assumption with the more general condition that this set forms a totally real subset of $M$, a similar approach can be employed to show that the continuous extension of $\eta$ to $M$ remains valid.

%%%%%%%%%%%%%%%%%%%%%%%%%%%%%%%%%%%%%%%%%%%%%%%%%%%%%%%%%%%%%%%%%%%%%%%%%%%%%%%%%%%%%

%\no\textbf{Acknowledgement:} The author would like to thank Viêt-Anh Nguyên and Kaushal Verma for giving their valuable feedback on the preprint version of this article.

%%%%%%%%%%%%%%%%%%%%%%%%%%%%%%%%%%%%%%%%%%%%%%%%%%%%%%%%%%%%%%%%%%%%%%%%%%%%%%%%%%%%

\end{document}